\documentclass[11pt, letterpaper]{amsart}

\usepackage[T1]{fontenc}

\usepackage{enumerate, amsmath, amsfonts, amssymb, amsthm, mathrsfs, complexity, wasysym, graphics, graphicx, xcolor, url, hyperref, hypcap, shuffle, xargs, multicol, overpic, pdflscape, multirow, hvfloat, minibox, accents, array, multido, xifthen, a4wide, xspace, ae, aecompl, blkarray, pifont, mathtools, etoolbox, dsfont, verbatim, stackrel, stmaryrd, xcolor, libertine, dynkin-diagrams, nicefrac, thmtools}

\usepackage{marginnote}
\hypersetup{colorlinks=true, citecolor=darkblue, linkcolor=darkblue,urlcolor=darkblue}
\usepackage[noabbrev,capitalise]{cleveref}

\usepackage[letterpaper,margin=1in]{geometry}
\usepackage{amsmath,amssymb,amsthm}
\usepackage{tikz,color,pgf,graphicx,url}
\usepackage{subcaption}
\usepackage{enumerate}
\usepackage{paralist}
\usepackage{enumitem}
\usepackage{thmtools}
\usepackage{thm-restate}

\definecolor{darkblue}{rgb}{0,0,0.7} 
\definecolor{darkred}{rgb}{0.7,0,0} 
\definecolor{green}{RGB}{57,181,74} 
\definecolor{violet}{RGB}{147,39,143} 

\newtheorem{theorem}{Theorem}

\newtheorem{lemma}[theorem]{Lemma}
\newtheorem{proposition}[theorem]{Proposition}

\newtheorem*{theorem*}{Theorem}

\theoremstyle{definition}

\newtheorem{example}[theorem]{Example}

\newcommand{\helena}[1]{{\bf \textcolor{orange!90}{Helena: #1}}}
\newcommand{\alex}[1]{{\bf \textcolor{blue!50}{Alex: #1}}}

\usepackage[normalem]{ulem}

\usepackage{xcolor} 
\usepackage{todonotes} 
\usepackage{hyperref}
\usepackage[capitalise]{cleveref}
\usepackage{graphicx}
\usepackage{xfrac} 

\usetikzlibrary{arrows.meta,angles,arrows,quotes,backgrounds}

\DeclareMathOperator{\N}{\mathbb{N}}
\DeclareMathOperator{\calG}{\mathcal{G}}

\renewcommand{\phi}{\varphi}

\def\inst#1{$^{#1}$}

\begin{document}

\title{Subgraph-universal planar graphs for trees\footnote{
M.\ Scheucher was supported by DFG Grant SCHE~2214/1-1. 
H.\ Bergold was partially supported by DFG Research Training Group 'Facets of Complexity' (DFG-GRK~2434). A.~Wesolek was supported by the DFG under Germany’s Excellence Strategy – The Berlin Mathematics Research Center MATH+ (EXC-2046/1, project ID: 390685689).
}}

\date{}

\author[H. Bergold, V. Ir\v si\v c, R. Lauff, J. Orthaber, M. Scheucher, A. Wesolek]{
Helena Bergold\inst{1}
\and 
Vesna Ir\v si\v c\inst{2}
\and 
Robert Lauff\inst{3}
\and 
Joachim Orthaber\inst{4}
\and 
Manfred Scheucher\inst{3}
\and 
Alexandra Wesolek\inst{3}
}

\title{Subgraph-universal planar graphs for trees}

\date{}

\begin{abstract}
We show that there exists an outerplanar graph on $O(n^{c})$ vertices for $c = \log_2(3+\sqrt{10}) \approx 2.623$ that contains every tree on $n$ vertices as a subgraph.
This extends a result of Chung and Graham from 1983 who showed that there exist (non-planar) $n$-vertex graphs with $O(n \log n)$ edges that contain all trees on $n$ vertices as subgraphs and a result from Gol'dberg and Livshits from 1968 who showed that there exists a universal tree for $n$-vertex trees on $n^{O(\log(n))}$ vertices.

Furthermore, we determine the number of vertices needed in the worst case for a planar graph to contain three given trees as subgraph to be on the order of $\frac{3}{2}n$, even if the three trees are caterpillars. This answers a question recently posed by Alecu et al. in 2024. 

Lastly, we investigate (outer)planar graphs containing all (outer)planar graphs as subgraph, determining exponential lower bounds in both cases. We also construct a planar graph on $n^{O(\log(n))}$ vertices containing all $n$-vertex outerplanar graphs as subgraphs.
\end{abstract}

\maketitle


\begin{center}
{\footnotesize
\inst{1} 
Technical University of Munich, Germany \\
TUM School of Computation, Information and Technology,
Department of Computer Science, Germany \\
Freie Universit\"at Berlin, Germany\\
\texttt{\{helena.bergold\}@tum.de}
\\\ \\
\inst{2} 
Faculty of Mathematics and Physics, University of Ljubljana, Slovenia\\
Institute of Mathematics, Physics and Mechanics, Ljubljana, Slovenia\\
\texttt{vesna.irsic@fmf.uni-lj.si}
\\\ \\
\inst{3} 
Institut f\"ur Mathematik, Technische Universit\"at Berlin, Germany\\
\texttt{\{lauff, scheucher, wesolek\}@math.tu-berlin.de}
\\\ \\
\inst{4} 
Institute of Software Technology, Graz University of Technology, Austria\\
\texttt{orthaber@ist.tugraz.at}
}
\end{center}



\subsection*{Acknowledgments} H.~Bergold was partially supported by the DFG Research Training Group 
`Facets of Complexity' (DFG-GRK~2434). V.~Ir\v{s}i\v{c} was supported by the Slovenian Research and Innovation Agency (ARIS) under the grants Z1-50003, P1-0297, and N1-0285, and by the European Union (ERC, KARST, 101071836). J.~Orthaber was supported by the Austrian Science Fund (FWF) grant W1230. M.~Scheucher was supported by the DFG Grant SCHE~2214/1-1. A.~Wesolek was supported by the Deutsche Forschungsgemeinschaft (DFG, German Research Foundation) under Germany’s Excellence Strategy – The Berlin Mathematics Research Center MATH+ (EXC-2046/1, project ID: 390685689).

We would like to thank Louis Esperet and Cyril Gavoille for insightful remarks on the paper and Stefan Felsner for helpful discussions and pointing us towards reference~\cite{TutteCensus1962}. 




\section{Introduction}
\label{sec:intro}
Nash-Williams' tree-covering theorem implies that every planar graph contains three spanning trees that cover all edges.
A natural question to ask is whether the converse is true, that is, whether any three trees on $n$ vertices are contained in some planar graph on $n$ vertices (as discussed by Alecu et al.~\cite{alecu2022treewidth}).
We show that this is false even when the three trees are caterpillars and the planar graph has $cn$ vertices for $c<\frac{3}{2}$.
On the other hand, every three trees on $n$ vertices are simultaneously contained in some planar graph on $\frac{3}{2}n$ vertices.

While from a structural perspective it is interesting to determine how to design a small planar graph that contains three given trees, the motivation of embedding multiple trees on $n$ vertices simultaneously in a (planar) graph has a long history, which originates in VLSI design.
For example, Valiant~\cite{valiant1976universal,MR0605722} used universal graphs to construct universal circuits.
A host graph $H$ is called \emph{(subgraph-)universal} for a family of graphs $\calG$
if every graph $G \in \calG$ appears as a subgraph of~$H$. 
Already in 1978,
Chung and Graham~\cite{ChungGraham1983}
showed that
an $n$-vertex graph that is subgraph-universal for $n$-vertex trees and is edge-minimal has $\Theta(n \log n)$ edges. Comparing the number of vertices to edges, their construction is far from planar and in fact the graph contains cliques of size $\Omega(\log(n))$. 
Finding small universal graphs with more restricting criteria is much harder. For host graphs that are trees, the first result from Gol'dberg and Livshits~\cite{gol1968minimal} from 1968 shows that there exists a universal tree for $n$-vertex trees on $n^{O(\log(n))}$ vertices.
Their construction was shown to be tight up to a polynomial factor in $n$ by Chung, Graham and Coppersmith~\cite{ChungGrahamCoppersmith1981} and they computed the right asymptotics up to the polynomial factor.
Several other graph classes are known to be universal for trees. It is known that hypercubes, expanding graphs and antiregular graphs are all universal for trees \cite{Wu1985, FriedmanPippenger1987, Merris2003}. 

In 1982, Babai et al.~\cite{BabaiChungErdösGrahamSpence1982} started the systematic study of universal graphs for sparse graph classes.
They gave a lower bound of $\Omega(\frac{n^2}{\log n})$ on the number of edges of a universal graph for the class of graphs with $cn$ edges for $c>1$, and they showed that asymptotically much less than the trivial upper bound $\binom{n}{2}$ is needed. 
Regarding planar graphs, they showed that there is a universal graph with $O(n^{3/2})$ edges.
After several improvements on the upper bound~\cite{BabaiChungErdösGrahamSpence1982,BonamyShorter2022}, recently Esperet et al.~\cite{Esperet2023Sparse} constructed a universal graph for all $n$-vertex planar graphs on $n^{1+o(1)}$ edges and $O(n)$ vertices by improving a construction of Dujmović et al.~\cite{DujmoviEtAl2021} with $n^{1+o(1)}$ edges and vertices. Not much has been previously known when restricting the host graph to be planar. A trivial upper bound on the number of vertices of a minimal planar universal graph for $n$-vertex planar graphs is the number of triangulations on $n$ vertices times $n$, which is $O((256/27)^n)$ as shown by Tutte~\cite{TutteCensus1962}. We give the first exponential lower bound. Previous research focused on universality for infinite graphs. Pach~\cite{PachProblem1981} showed that there is no countable planar graph containing all countable planar graphs, answering a question of Ulam. Recently, Huynh et al.~\cite{huynh2021universality} strengthened this result by showing that every countable graph that contains all countable planar graphs contains an infinite clique as a minor. 

\subsection{Our contribution.}
We continue the line of research from \cite{ChungGraham1983} and \cite{ChungGrahamCoppersmith1981} and study planar host graphs for trees.
As our main result, we construct planar graphs with polynomially many vertices that are subgraph-universal for all $n$-vertex trees.
In fact, our constructions are even outerplanar.

\begin{restatable}{theorem}{UniversalOuterplanar}\label{thm:polynomial_bound}
For every $n \in \mathbb{N}$, there exists an outerplanar graph on $O(n^{c})$ vertices, where $c = \log_2(3 + \sqrt{10}) \approx 2.623$, that contains all $n$-vertex trees as subgraphs.
\end{restatable}
Our result shows that imposing strong structural properties on the host graph does not necessarily blow up the number of its vertices (as is the case when restricting the host graph to be a tree).
On the other side, we show that there are no subexponential sized planar graphs that contain all $n$-vertex planar graphs.

\begin{restatable}{theorem}{planarlowerbound}
\label{thm:planarlowerbound}
    A planar universal graph for $n$-vertex planar graphs has $\Omega(n^{-3/2} (27/4)^n)$ vertices.
\end{restatable}

We also show an exponential lower bound on the number of vertices of an outerplanar universal graph for $n$-vertex outerplanar graphs and construct for this class a planar universal graph with $n^{\theta(\log(n))}$ vertices.
The paper is organized as follows. Theorem~\ref{thm:polynomial_bound} is proven in Section~\ref{sec:universal_construction}.
The main idea of the proof is to use stacked triangulations, which was inspired by a computer search that is presented in Section~\ref{sec:computer}.
The converse of the Nash-Williams' tree-covering theorem is discussed in Section~\ref{sec:3_trees}.
In Section~\ref{sec:planaruniversal} we discuss (outer)planar universal graphs containing all $n$-vertex (outer)planar graphs. Some directions for further research are given in Section~\ref{sec:concluding}. 

\subsection{Further related work} A graph $H$ is called \emph{induced-universal} if every graph \mbox{$G \in \calG$} appears as an induced subgraph of $H$. Induced-universal graphs are important for designing adjacency labeling schemes.  An \emph{adjacency labeling scheme} for a family of graphs $\calG$ is an assignment of labels to the vertices
 of each graph in $\calG$ such that given the labels of two vertices in the graph it is possible to determine whether or not the vertices are adjacent in the graph. Kannan, Naor and Rudich~\cite{MR1186827} showed that a family $\calG$ of graphs has a $k$-bit adjacency labeling scheme with unique labels if and only if there is an induced-universal graph for $\calG$ on at most $2^k$ vertices.

Already in 1964, Rado \cite{Rado1964} constructed a countably infinite graph that is induced-universal for all graphs, that is, each finite graph appears as an induced subgraph.
Only one year later, Moon \cite{Moon1965} showed that the minimum number of vertices of an induced-universal graph for all $n$-vertex graph has size at least $\sqrt{2}^{n-1}$ and Alon \cite{MR3613451} showed this is asymptotically tight.
In subsequent literature, various subclasses of graphs have been investigated intensively such as  
graphs \cite{BabaiChungErdösGrahamSpence1982}, 
planar graphs \cite{DujmoviEtAl2021,Esperet2023Sparse}, 
(planar) graphs with bounded degree \cite{BhattSandeepChungLeightonRosenberg1989}, 
trees 
and
caterpillars \cite{ChungGrahamShearer1981,ChungGraham1983}.
We refer the interested reader to the
Master's thesis of Christian Kouekam \cite[Table 6.1]{KouekamMaster21}. 
It summarizes known bounds on the minimum number of vertices of a graph that is induced-universal for various graph classes.

\section{A universal outerplanar graph for trees}
\label{sec:universal_construction}

The goal of this section is to construct, for every $n\in\N$, an (outer)planar graph $G$ on as few vertices as possible such that all $n$-vertex trees are contained in $G$ as subgraphs.
We will see in the following that a number of vertices polynomial in $n$ is sufficient for $G$.

As a starting point, we inductively define the \emph{uniformly stacked triangulation}\footnote{Stacked triangulations are also known as \emph{planar 3-trees} or \emph{Apollonian networks} in the literature.} $S_d$ in the plane with \emph{depth} $d \in \mathbb{N}$ as follows:
$S_0$ is a triangle determined by three fixed points $o_1,o_2,o_3$ not lying on a common line.
For $d > 0$, we define $S_d$ by subdividing each bounded triangular face $D$ of $S_{d-1}$ into three triangles by placing a new point in the interior of~$D$ and connecting it to the three vertices of~$D$.
See \Cref{fig:S3} for the uniformly stacked triangulation $S_3$.
By construction, $S_d$ has $3^d$ bounded faces and $3 + 3^0 + 3^1 + \dots + 3^{d-1} = \frac{3^d+5}{2}$ vertices. 
Each of the three vertices $o_1,o_2,o_3$, which we call the \emph{outer vertices} of $S_d$, is adjacent to exactly $2^d-1$ interior vertices of~$S_d$. 
Since we deal with several copies and substructures, we denote by $o_i(S_d)$ for $i =1,2,3$ the outer vertex $o_i$ of the copy of $S_d$.

\begin{figure}[htb]
\centering
\subcaptionbox{\label{fig:S3}}[.49\textwidth]{\includegraphics[page=1]{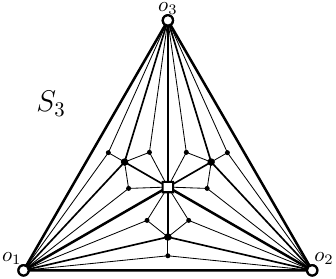}}
\subcaptionbox{\label{fig:Sd-3}}[.49\textwidth]{\includegraphics[page=2]{figures/stacked-new.pdf}}
\caption{\textsc{(A)}~The uniformly stacked triangulation $S_3$. \textsc{(B)}~The highlighted faces represent copies of $S_{d-3}$ that are incident to the center vertex in $S_d$.} 
\label{fig:uniform-stacked}
\end{figure}

For $d \geq 1$, we define the \emph{center vertex} $c  = c(S_d)$ of $S_d$ as the vertex that is part of $S_1$ but not~$S_0$.
Due to the recursive definition, each of the three triangles $\Delta_{c, o_1, o_2}$, $\Delta_{c, o_2, o_3}$, and $\Delta_{c, o_3, o_1}$ in $S_d$ hosts a copy of~$S_{d-1}$. 
Consequently, for $1 \leq i \leq d$, the center vertex is incident to $3 \cdot 2^{i-1}$ copies of $S_{d-i}$ in~$S_d$; see \Cref{fig:Sd-3} for an example with $i=3$. 

As an introductory example, we show that uniformly stacked triangulations on a polynomial number of vertices contain all $k$-ary trees on at most $n$ vertices.
Recall that a \emph{$k$-ary tree of height $h$} is a rooted tree where each vertex has exactly $k$ children except for the vertices on level $h$, which are all leaves.
The root is considered to be at level~$0$.

\begin{proposition}\label{prop:k-ary}
Every $k$-ary tree with height at most $\log_k(n)$ can be embedded into $S_d$ with depth 
$d =2\log_2(n)$.
\end{proposition}

\begin{proof}
First note that the central vertex of a copy of $S_d$ is adjacent to $3\cdot 2^{s} $ copies of $S_{d-1-s}$.

We show by induction that we can embed a $k$-ary tree of height $h = \log_k(n)$ into the interior vertices of a stacked triangulation $S_d$ for $d=h\cdot  (\lceil \log_2(k/3)\rceil+1)+1$ such that the root is the central vertex of $S_d$.
The induction is on $h$. Consider $s=\lceil \log_2(k/3) \rceil +1$ and therefore $k \leq 3\cdot 2^{s-1} $. Certainly we can embed a $k$-ary tree of height $h=0$ into the interior of $S_1$ such that the root is the central vertex. We assume now $h\geq 1$. Since the central vertex of $S_d=S_{hs+1}$ is incident to $3\cdot 2^{s-1}\geq k$ copies of $S_{(h-1)s+1}$, we can embed by induction all the $k$-ary trees of height $h-1$ of $S_d$ into the interior of the subcopies $S_{(h-1)s+1}$ such that the root is the central vertex of $S_{(h-1)s+1}$. 
We add the central vertex of $S_d$ to the graph and add its edges to the central vertices of the $S_{(h-1)s+1}$ subcopies. This gives an embedding of the $k$-ary tree of height $h$ into $S_d$ such that all its vertices are in the interior and the root is the central vertex. This proves the theorem by using the following inequality 
\begin{align*}
    d & = (\lceil \log_2(k/3) \rceil +1) \log_k(n)+1 \leq 2\log_2(n)+1. \hfill \qedhere
\end{align*}
\end{proof}

General trees $T$ are not as symmetric.
However, we can always find a vertex in $T$ that splits the tree into somewhat balanced subtrees.

\begin{lemma}[\cite{Jordan1869}]\label{lem:good root}
Every tree $T$ on $n$ vertices contains a vertex $J$ such that all components of $T-J$ contain at most $\frac{n}{2}$ vertices.
\end{lemma}

We call such a vertex $J$ a \emph{Jordan separator} of $T$.
Using the Jordan separator, we mimic the embedding of the $k$-ary trees.
We place $J$ at the center vertex of the stacked triangulation and embed every subtree into the interior of a subtriangulation.
Recursively splitting each subtree again with a Jordan separator, we can decompose $T$ into components of constant size with $\log_{2}(n)$ such steps. 

For embedding $T$ into a stacked triangulation, note that in each subtree $T'$ we have a \emph{marked vertex} that needs to be adjacent to~$J$.
In the embedding, $T'$ will be inside of some subtriangulation and $J$ will be mapped to one of its outer vertices.
After four recursive steps of splitting $T$ with a Jordan separator, we might have a subtree with four marked vertices that need to be adjacent to a different vertex each, which has been embedded at a previous step.
However, every subtriangulation only has three outer vertices.
Therefore, we introduce an intermediate splitting step which follows from the fact that trees are median graphs.
\begin{lemma}[{\cite[Proposition 1.26]{imrich2000product}}]\label{lem:separate-3-vertices}
For every triple of vertices $a,b,c$ in a tree $T$ there exists a unique vertex $M$ in $T$ that lies on all three paths connecting $a$, $b$, and $c$ in $T$.
\end{lemma}

\begin{figure}[htb]
\centering
\subcaptionbox{\label{fig:stacked-S-1}}[.49\textwidth]{\includegraphics[page=3]{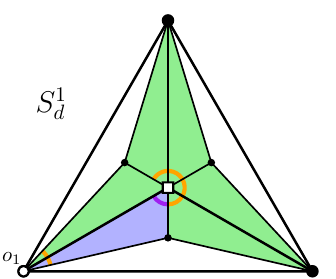}}
\subcaptionbox{\label{fig:stacked-S-2}}[.49\textwidth]{\includegraphics[page=4]{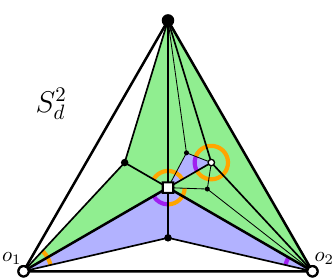}}
\caption{The stacked triangulations \textsc{(A)}~$S_d^1$ and \textsc{(B)}~$S_d^2$. Subtriangulations $S_{d-1}^1$ and $S_{d-1}^2$ are represented by green and blue triangles, respectively; white triangles are empty. The outer vertices $o_1$ (and $o_2$) of the subtriangulations are marked with orange (and purple) sectors.}
\label{fig:stacked-construction}
\end{figure}

Using \Cref{lem:separate-3-vertices}, we can split a subtree with three marked vertices at the \emph{marker separator} $M$. This ensures that after the splitting the number of marked vertices in each subtree is again at most two.
We now define two stacked triangulations $S_d^1$ and $S_d^2$ such that any tree on at most $2^d$ vertices and with at most $i$ marked vertices $m_j$ can be embedded into the interior of $S_d^i$ such that $m_j$ is adjacent to $o_j(S_d^i)$, for $i=1,2$.
Instead of using the uniformly stacked triangulation, we recursively build our triangulation out of two types of triangulation $S_{d-1}^1$ and $S_{d-1}^2$. 
As the base case, both $S_0^1$ and $S_0^2$ are equal to a triangle with one vertex stacked inside, that is, the uniformly stacked triangulation~$S_1$.
\Cref{fig:stacked-construction} gives an illustration of the constructions. 
Triangulation $S_d^1$ consists of one copy of $S_{d-1}^2$ between the outer vertex $o_1(S_d^1)$ and the center vertex $c(S_d^1)$, and five copies of $S_{d-1}^1$, four of which have their outer vertex $o_1$ at the center vertex $c(S_d^1)$.
In $S_d^2$ there is one copy of $S_{d-1}^2$ between $o_1(S_d^2)$ and the center vertex $c(S_d^2)$, another between $c(S_d^2)$ and $o_2(S_d^2)$, and a third between $c(S_d^2)$ and a vertex in the interior of $S_d^2$, which we call the \emph{center-right} vertex.
Furthermore, there are eight copies of $S_{d-1}^1$, for three of them their outer vertex $o_1$ coincides with the center vertex $c(S_d^2)$ and for four of them it coincides with the center-right vertex.

Letting $a_i(d)$ be the number of interior vertices of $S_d^i$, we get the system of recurrences
\begin{align*}
a_1(d) &= 5 a_1(d-1) + a_2(d-1) + 4\\
a_2(d) &= 8 a_1(d-1) + 3 a_2(d-1) + 6.
\end{align*}
Solving this system tells us that $S_d^i$ has $O(7^d)$ many vertices.

Another important part of the constructions is that there is one copy of $S_{d-1}^1$ with $o_1(S_d^i)=o_1(S_{d-1}^1)$, for $i=1,2$. 
Hence, $o_1(S_d^i)$ coincides with $o_1$'s of two smaller copies, $S_{d-1}^1$ and  $S_{d-1}^2$. 
Together with the next lemma, this ensures that we can embed multiple trees of small size into one copy of $S_d^i$. We embed the tree inductively by splitting $T$ along the Jordan separator and splitting $S_d^i$ into appropriately sized regions such that the subtrees can be embedded in their assigned regions by induction. The next lemma tells us how we can split $S_d^i$ into these regions.

We say that two embedded graphs are internally disjoint if their common vertices are on the outer face of both graphs. Let $\mathcal{S}_d=\{S_d^1,S_d^2\}$.
\begin{lemma}
\label{lem:embedding_lemma}
Let $\ell,d,d_1,\ldots,d_k$ be nonnegative integers with $d_j \leq d-1$ for all $j=1,\ldots,k$ and
\begin{equation*}
\sum_{j=1}^k2^{d_j}\le 2^{d+\ell}.
\end{equation*}
Then $2^\ell$ internally disjoint triangulations which are each in $\mathcal{S}_d$ with common outer vertex $o_1$ contain pairwise internally disjoint copies $S_{d_j}^{i_j}$ for $i_j \in \{1,2\}$ and $j\in [k]$ with outer vertex at $o_1$. 
\end{lemma}

\begin{proof}
Since each copy of $S_d^i$ contains internally disjoint copies of $S_{d-1}^1,S_{d-1}^2$ with the same outer vertex $o_1$, they contain exactly $2^{j}$ pairwise internally disjointed copies of graphs in $\mathcal{S}_{d-j}$ each, all having the same outer vertex $o_1$, for every $j=1,\ldots,d$. 
Therefore it is natural to give a copy of a graph in $\mathcal{S}_{d_i}$ a cost of $2^{d_i}$ and say that we have a budget of $2^{d+1}$ in total. 
We greedily assign for $i=1,\ldots,k$ a copy of a graph in $\mathcal{S}_{d_i}$, going over the $d_i$ in order of descending size. This is to avoid a small one blocking a large one we need later. Then we have a cost of
\begin{equation*}
\sum_{i=1}^k 2^{d_i}\le 2^{d+1}
\end{equation*}
and hence we do not exceed the budget.
\end{proof}

We are now ready to prove a slightly weaker version of \Cref{thm:polynomial_bound}.

\begin{theorem}\label{thm:embed-to-triangulation}
Let $T$ be a tree on $n \leq 2^d$ vertices for $d\in\N$, with up to two, not necessarily distinct, marked vertices $m_1$ and $m_2$.
Then $T$ can be embedded into the interior of $S_d^1$ such that $m_1$ is adjacent to $o_1$ of $S_d^1$ and $T$ can be embedded into the interior of $S_d^2$ such that $m_1$ and $m_2$ are adjacent to $o_1$ and $o_2$ of $S_d^2$, respectively.
\end{theorem}

\begin{proof}
We give a proof by induction on the depth $d$, simultaneously for the stacked triangulations $S_d^1$ and~$S_d^2$.
For $d=0$, the only tree on at most one vertex is mapped to the center vertex of $S_0^1 = S_0^2 = S_1$, which is adjacent to all outer vertices.

For the inductive step, let $J$ be a Jordan separator of $T$, which exists by \Cref{lem:good root}, and let $T_1,\ldots,T_k$ be the connected components of $T - J$, which are subtrees of~$T$.
Furthermore, let $d_j=\lceil\log_2(|T_j|)\rceil$ for $j=1,\ldots,k$.
Then $d_j$ is the smallest integer such that $|T_j| \leq 2^{d_j}$ and we have $d_j \leq d-1$, since $J$ is a Jordan separator.
In each subtree $T_j$, we mark the vertex $r_j$ adjacent to~$J$.
We distinguish three cases.

\begin{figure}[htb]
\centering
\subcaptionbox{\label{fig:split-1-marked}}[.49\textwidth]{\includegraphics[page=1]{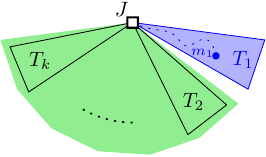}}
\subcaptionbox{\label{fig:embed-1-marked}}[.49\textwidth]{\includegraphics[page=5]{figures/stacked-new.pdf}}
\caption{\textsc{(A)}~A splitting into subtrees with only one marked vertex $m_1$ in $T$. \textsc{(B)}~Embedding the subtrees into $S_d^1$.}
\label{fig:1-marked}
\end{figure}

\begin{description}[leftmargin=0em]

\item[Case 1: $T$ contains at most one marked vertex~$m_1$] 
We map $J$ to the center vertex~$c$ of~$S_d^i$, for $i \in \{ 1, 2 \}$.

If $m_1=J$ or there is no marked vertex, then the required adjacency is fulfilled.
Otherwise we assume, without loss of generality, that $m_1 \in T_1$.
In both $S_d^1$ and $S_d^2$ there are at least four copies of $S_{d-1}^1$ incident to $c$ with their outer vertex $o_1$ equal to~$c$.
We can apply \Cref{lem:embedding_lemma} since the sizes of the subtrees are bounded by $2^{d-1}$ as $J$ is a Jordan separator and 
\begin{equation*}
\sum_{j=2}^k2^{d_j}\le\sum_{j=2}^k2|T_j|\le 2^{d+1}.
\end{equation*}

By the induction hypothesis, we can simultaneously embed the trees $T_2,\ldots,T_k$ into the interior of those copies of $S_{d-1}^1$ such that $r_2,\ldots,r_k$ are all adjacent to~$J$. 
Furthermore, by the induction hypothesis, we can embed $T_1$ into the copy of $S_{d-1}^2$ such that $r_1$ is adjacent to $J$ and $m_1$ is adjacent to $o_1$, as claimed.
\Cref{fig:1-marked} illustrates the case for~$S_d^1$, and as explained above, the case for $S_d^2$ is similar.\\

\item[Case 2: $m_1$ and $m_2$ are not in the same subtree $T_j$]
Then every subtree $T_1, \ldots, T_k$ contains at most two marked vertices.
We again map $J$ to the center vertex $c$ of $S_d^2$.
If $m_1=J$ or $m_2=J$, then similarly as in Case 1 the required adjacencies are fulfilled. 
Hence we assume, without loss of generality, that $m_1\in T_1$ and $m_2\in T_2$.
As $d_1,d_2\le d-1$ we can embed $T_1$ and $T_2$ into the copies of $S_{d-1}^2$ marked blue and red in \Cref{fig:embed-2-marked} such that $m_1$ and $m_2$ are adjacent to $o_1$ and $o_2$, respectively, and both $r_1$ and $r_2$ are adjacent to $J$.
Furthermore, we can embed the remaining subtrees into the copies of $S_{d_j}^i$ adjacent to $J$ such that the marked vertices $r_3,\ldots,r_k$ are adjacent to $J$. This is again an application of \Cref{lem:embedding_lemma}, which we can apply because the sizes of the subtrees are indeed bounded by $2^{d-1}$ as $J$ is a Jordan separator and we have

\begin{equation*}
\sum_{j=3}^k2^{d_j}\le\sum_{j=3}^k2|T_j|\le 2^{d+1}.
\end{equation*}

\begin{figure}[hbt]
\centering
\subcaptionbox{\label{fig:split-2-marked}}[.49\textwidth]{\includegraphics[page=2]{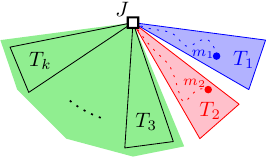}}
\subcaptionbox{\label{fig:embed-2-marked}}[.49\textwidth]{\includegraphics[page=6]{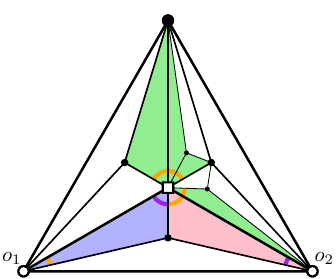}}
\caption{\textsc{(A)}~A splitting of $T$ with Jordan separator such that $m_1$ and $m_2$ are in different subtrees. \textsc{(B)}~Embedding the subtrees into $S_d^2$.}
\label{fig:2-marked}
\end{figure}

\begin{figure}[htb]
\centering
\subcaptionbox{\label{fig:split-3-marked}}[.49\textwidth]{\includegraphics[page=3]{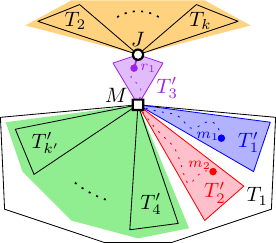}}
\subcaptionbox{\label{fig:embed-3-marked}}[.49\textwidth]{\includegraphics[page=7]{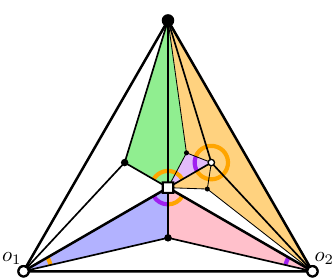}}
\caption{\textsc{(A)}~A splitting of $T$ with a Jordan separator $J$ such that $m_1$ and $m_2$ are in the same subtree $T_1$, which is further split at~$M$. \textsc{(B)}~Embedding the subtrees into $S_d^2$.}
\label{fig:3-marked}
\end{figure}

\item[Case 3: $m_1$ and $m_2$ are in the same subtree $T_1$]
Then $T_1$ contains three marked vertices $m_1$, $m_2$, and $r_1$.
Hence, using \Cref{lem:separate-3-vertices} and introducing the vertex $M$, we further split $T_1$ into subtrees $T'_j$, for $j=1,\ldots,k'$, and let $d'_j=\lceil\log_2(|T'_j|)\rceil$.
This time we map $M$ to the center vertex $c$ and $J$ to the center-right vertex~$c_r$.
We embed the subtrees $T_2,\ldots T_k$ into the four copies of $S_{d-1}^1$ incident to~$c_r$, which is possible by \Cref{lem:embedding_lemma}.
Further, we assume that $m_1$, $m_2$, and $r_1$ are contained in $T'_{1}$, $T'_{2}$, and $T'_{3}$, respectively.
If any of them coincide with $M$, then the corresponding subtree is empty and the required adjacencies are fulfilled by~$M$.
In each subtree, we mark the vertex $r'_j$ that is the neighbor of $M$ in that subtree.
Since $\sum 2^{d'_j} \leq 2^d$, we can embed the subtrees $T'_4,\ldots,T'_{k'}$ into two copies of $S_{d-1}^1$ incident to $c$ by \Cref{lem:embedding_lemma} and the induction hypothesis.
The following embeddings are ensured by the induction hypothesis:
\begin{itemize}
\item $T'_1$ is embedded into the copy of $S_{d-1}^2$ between $o_1$ and $c$.
\item $T'_2$ is embedded into the copy of $S_{d-1}^2$ between $c$ and $o_2$.
\item $T'_3$ is embedded into the copy of $S_{d-1}^2$ between $c$ and $c_r$.\qedhere
\end{itemize}
\end{description}

\end{proof}

Given a tree, we can mark any vertex, and then use Theorem~\ref{thm:embed-to-triangulation}.
Since we can embed every tree on $n\le 2^d$ vertices into $S_d^i$, this is a planar graph on $O(n^{\log_2(7)})$ vertices that contains all $n$-vertex trees as subgraph.
To reach the statement of \Cref{thm:polynomial_bound}, it remains to make the construction outerplanar and slightly reduce the number of vertices in it.

\UniversalOuterplanar*

\begin{proof}[Proof sketch.]
We build on the constructions $S_d^1$ and $S_d^2$ from \Cref{thm:embed-to-triangulation}.

First note that we embed all subtrees into the interior of those stacked triangulations.
Hence, we only need the outer vertices $o_1$ and $o_2$, the center vertex $c$, and the center-right vertex $c_r$.
Further, we need the edges between $c$ and the other three mentioned vertices.
Removing all other vertices and edges from $S_d^1$ and $S_d^2$ results in outerplanar graphs that still fulfill all properties needed for the embedding; see \Cref{fig:outerplanar} for an illustration.

\begin{figure}[htb]
\centering
\subcaptionbox{\label{fig:outerplanar-S1}}[.49\textwidth]{\includegraphics[page=8]{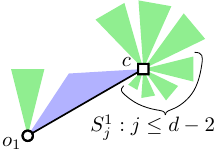}}
\subcaptionbox{\label{fig:outerplanar-S2}}[.49\textwidth]{\includegraphics[page=9]{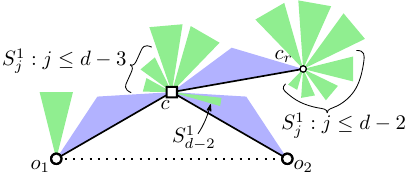}}
\caption{The outerplanar constructions corresponding to \textsc{(A)}~$S_d^1$ and \textsc{(B)}~$S_d^2$.}
\label{fig:outerplanar}
\end{figure}

Second, every time that we use \Cref{lem:embedding_lemma} to embed multiple subtrees incident to one vertex, the sizes of the needed subtriangulations actually sum up to at most $2^{d+1}-1$, because at least the separator vertex $J$ (or $M$) is not part of any subtree.
Therefore, we can replace one of the copies of $S_{d-1}^1$ incident to $c$ and $c_r$, each, by one copy of $S_j^1$ for every $j \leq d-2$.
Furthermore, within $S_d^2$ in the copy of $S_{d-1}^2$ between $c$ and $o_2$ there is a copy of $S_{d-2}^1$ incident to $c$ that is unused in all cases.
Hence, we only need to add copies of $S_j^1$ up to $j = d-3$ there.

Implementing these changes, we end up with the system of recurrences
\begin{align*}
a_1(d) &= 4 a_1(d-1) + a_2(d-1) + 1 + \sum_{j=0}^{d-2} a_1(j)\\
a_2(d) &= 6 a_1(d-1) + 3 a_2(d-1) + 2 + a_1(d-2) + 2 \sum_{j=0}^{d-3} a_1(j).
\end{align*}
Solving this system yields the asymptotic bound of $O((3+\sqrt{10})^d) = O(n^{\log_2(3+\sqrt{10})})$ as claimed.
\end{proof}

\section{Embedding three trees into a planar graph}
\label{sec:3_trees}

In this section we focus on the converse of the Nash-Williams' tree-covering theorem: given any three trees on $n$ vertices we consider the smallest number of vertices of a planar graph containing them as subgraph. Recall the following definitions.

Let $n,m$ be positive integers and denote $[n] = \{1,\ldots,n\}$. 
A \emph{leaf} of a tree is a vertex of degree 1. A \emph{caterpillar} is a tree such that the subgraph induced on all vertices of degree two or more is a path. This subgraph is called the \emph{spine} of the caterpillar. Let $n, k \in \mathbb{N}$ and $k \mid n$. A \emph{$k$-caterpillar on $n$ vertices} $C_{n,k}$ is a caterpillar with spine $P_k$ such that every vertex on the spine is adjacent to $ \frac{n}{k}-1$ leaves. In particular this means that $|V(C_{n,k})| = n$ and that if $u \in V(C_{n,k})$ is not a leaf, then $|N(u)| \in \{\frac{n}{k}, \frac{n}{k}+1\}$ for $k >1$. A \emph{star} is $C_{n, 1}$ for some $n\geq 2$ which is isomorphic to $K_{1,n-1}$. Similarly, $C_{n, 2}$ is also called a \emph{double star}.

\begin{theorem}
    \label{thm:counterexample}
    If $n \geq 6 k^2$, then there is no planar graph on $n$ vertices containing $C_{n,1}$, $C_{n, 2}$, and $C_{n, k}$, $k\geq 5$, as subgraphs.
\end{theorem}

\begin{proof}
Let $G$ be a subgraph-universal planar graph for $\{C_{n,1},C_{n, 2},C_{n, k}\}$ on $n$ vertices. There are only two ways (up to isomorphism) in which $C_{n,1}$ and $C_{n,2}$ can be contained in $G$, either the center of the star $C_{n,1}$ is mapped to the same vertex as one of the two spine vertices of $C_{n,2}$ or to a leaf vertex of $C_{n,2}$; see Figure \ref{fig:glueing}. The graph in the latter case is strictly contained in the first, so we can assume that the center of the star $C_{n,1}$ is mapped to the same vertex as one of the two spine vertices of $C_{n,2}$. Let $x$ be the shared spine vertex between $C_{n,1}$ and $C_{n,2}$ and let $y$ be the other spine vertex of $C_{n,2}$. Then we have $|N(x)\cap N(y)|=\frac{n-2}{2}$.

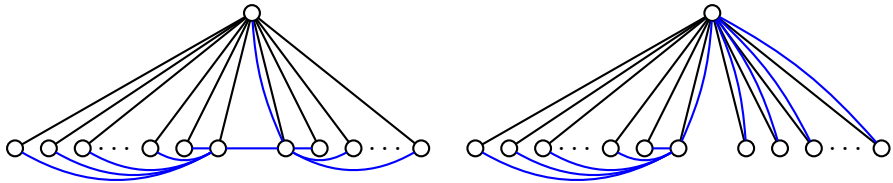
\begin{figure}[h!]
\centering
  \begin{tikzpicture}[thick, scale=.45]
    \tikzstyle{uStyle}=[shape = circle, minimum size = 6.0pt, inner sep = 0pt,
    outer sep = 0pt, draw, fill=white]
    \tikzstyle{lStyle}=[shape = rectangle, minimum size = 20.0pt, inner sep = 0pt,
outer sep = 2pt, draw=none, fill=none]
    \tikzset{every node/.style=uStyle}
    \begin{scope}
   \draw (8,3) node (v0){};
    \foreach \i in {1,...,3,5,6,7,9,10,11,13}
    \draw (\i,-1) node (v\i) {};
    \node[shape = circle, minimum size = 6.0pt, inner sep = 0pt,
    outer sep = 0pt, draw = none] at (4, -1) {$\dots$};
    \node[shape = circle, minimum size = 6.0pt, inner sep = 0pt,
    outer sep = 0pt, draw = none] at (12, -1) {$\dots$};
    \foreach \i in {1,...,3,5,6,7,9,10,11,13}
    \draw (v0) edge (v\i);
    \draw[blue] (v7) edge (v9);
    \draw[blue] (v7) edge (v6);
    \draw[blue] (v10) edge (v9);
    \draw[blue,bend left] (v7) edge (v5);
    \draw[blue,bend left] (v11) edge (v9);  
     \draw[blue,bend left] (v7) edge (v3);
    \draw[blue,bend left] (v13) edge (v9);  
     \draw[blue,bend left] (v7) edge (v2);
     \draw[blue,bend left] (v7) edge (v1); 
    \draw[blue,bend left=10] (v9) edge (v0);   
   \end{scope}
\end{tikzpicture}\hspace{1em}
  \begin{tikzpicture}[thick, scale=.45]
    \tikzstyle{uStyle}=[shape = circle, minimum size = 6.0pt, inner sep = 0pt,
    outer sep = 0pt, draw, fill=white]
    \tikzstyle{lStyle}=[shape = rectangle, minimum size = 20.0pt, inner sep = 0pt,
outer sep = 2pt, draw=none, fill=none]
    \tikzset{every node/.style=uStyle}
    \begin{scope}
   \draw (8,3) node (v0) {};
    \foreach \i in {1,...,3,5,6,7,9,10,11,13}
    \draw (\i,-1) node (v\i) {};
    \node[shape = circle, minimum size = 6.0pt, inner sep = 0pt,
    outer sep = 0pt, draw = none] at (4, -1) {$\dots$};
    \node[shape = circle, minimum size = 6.0pt, inner sep = 0pt,
    outer sep = 0pt, draw = none] at (12, -1) {$\dots$};
    \foreach \i in {1,...,3,5,6,7,9,10,11,13}
    \draw (v0) edge (v\i);
    \draw[blue] (v7) edge (v6);
    \draw[blue,bend left] (v7) edge (v5);  
     \draw[blue,bend left] (v7) edge (v3);;  
     \draw[blue,bend left] (v7) edge (v2);
     \draw[blue,bend left] (v7) edge (v1); 
    \draw[blue,bend left=10] (v0) edge (v9);  
    \draw[blue,bend left=10] (v0) edge (v7);  
    \draw[blue,bend left=10] (v0) edge (v10);  
    \draw[blue,bend left=10] (v0) edge (v11);    
    \draw[blue,bend left=10] (v0) edge (v13);    
   \end{scope}
\end{tikzpicture}
\caption{There are only two non-isomorphic possibilities of embedding both $C_{n,1}$ (edges in black) and $C_{n,2}$ (edges in blue) into a graph on $n$ vertices. }
\label{fig:glueing}
\end{figure}
When adding $C_{n,k}$, we see that at least $k-2$ spine vertices of $C_{n,k}$ are neither mapped to $x$ nor~$y$.
If any of them has at least 3 neighbours in $N(x)\cap N(y)$, then the graph contains a~$K_{3,3}$, which cannot happen in a planar graph. Hence they can have at most two neighbours in $N(x)\cap N(y)$.
From each of the corresponding $k-2$ stars $C_{\frac{n}{k}, 1}$, at most the center and two leaves can be mapped to $N(x)\cap N(y)$ and the remaining vertices are mapped to the set $[n]\setminus (N(x)\cap N(y))$ of size $\frac{n+2}{2}$. There are 
$(k-2) \cdot \frac{n}{k} - 3 \cdot (k-2)$ remaining vertices. In order to embed them in a planar way, it has to hold that
    \begin{equation*}
        (k-2) \cdot \frac{n}{k} \leq 3 \cdot (k-2)+ \frac{n+2}{2}.
    \end{equation*}
Using the relations $k\geq 5$ and $n\geq 6k^2$, we recover
    \begin{equation*}
        0\leq 3k-5+\frac{(4-k)n}{2k}\leq -5, 
    \end{equation*}
    a contradiction.
\end{proof}
We show in \Cref{ex:threetrees} that already for $T_1 = C_{48,1}$, $T_2 = C_{48,2}$, and $T_3 = C_{48,8}$ there is no planar graph containing all three as subgraphs.
\Cref{thm:counterexample} means that the answer to Question 3 posed in~\cite{alecu2022treewidth} is negative, where it was asked whether for any three $n$-vertex caterpillars there is an $n$-vertex universal planar graph. However, this raises the question of how many vertices are needed such that for any three trees there is a planar graph containing them as subgraphs.

\begin{theorem}
    \label{thm:upper-bound-3-trees}
For any three trees $T_1$, $T_2$ and $T_3$ on $n$ vertices each, there exists a subgraph-universal planar graph for $\{T_1,T_2,T_3\}$ on $\frac32 n$ vertices.
\end{theorem}
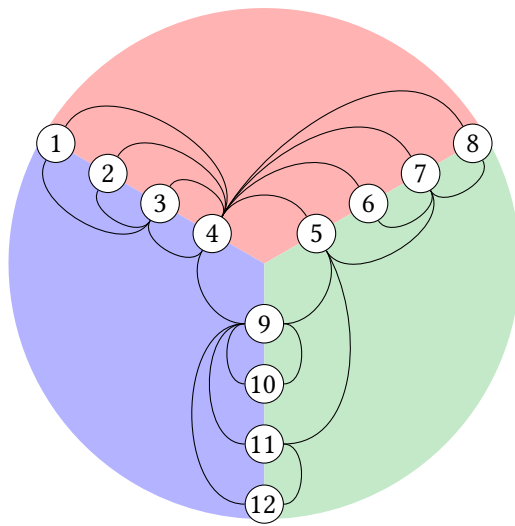
\begin{figure}
    \centering
    \begin{tikzpicture}[scale = 0.4]
        \node[draw,fill=white,inner sep=2pt,text width=2mm,align=center] (A) at (30:2cm) [circle] {5};
        \node[draw,fill=white,inner sep=2pt,text width=2mm,align=center] (B) at (30:4cm) [circle] {6};
        \node[draw,fill=white,inner sep=2pt,text width=2mm,align=center] (C) at (30:6cm) [circle] {7};
        \node[draw,fill=white,inner sep=2pt,text width=2mm,align=center] (D) at (30:8cm) [circle] {8};

        \node[draw,fill=white,inner sep=2pt,text width=2mm,align=center] (a) at (150:2cm) [circle] {4};
        \node[draw,fill=white,inner sep=2pt,text width=2mm,align=center] (b) at (150:4cm) [circle] {3};
        \node[draw,fill=white,inner sep=2pt,text width=2mm,align=center] (c) at (150:6cm) [circle] {2};
        \node[draw,fill=white,inner sep=2pt,text width=2mm,align=center] (d) at (150:8cm) [circle] {1};
        
        \node[draw,fill=white,inner sep=2pt,text width=2mm,align=center] (w) at (270:2cm) [circle] {9};
        \node[draw,fill=white,inner sep=2pt,text width=2mm,align=center] (x) at (270:4cm) [circle] {\hspace*{-1mm}10};
        \node[draw,fill=white,inner sep=2pt,text width=2mm,align=center] (y) at (270:6cm) [circle] {\hspace*{-1mm}11};
        \node[draw,fill=white,inner sep=2pt,text width=2mm,align=center] (z) at (270:8cm) [circle] {\hspace*{-1mm}12};

        \draw (A) to[out = -60, in = 0]     (w);
        \draw (A) to[out = -60, in = 0]     (y);
        \draw (A) to[out = -60, in = -60]   (C);
        \draw (B) to[out = -60, in = -60]   (C);
        \draw (C) to[out = -60, in = -60]   (D);
        \draw (w) to[out = 0  , in = 0]     (x);
        \draw (y) to[out = 0  , in = 0]     (z);

        \draw (a) to[out = 60, in = 120]    (A);
        \draw (a) to[out = 60, in = 120]    (B);
        \draw (a) to[out = 60, in = 120]    (C);
        \draw (a) to[out = 60, in = 120]    (D);
        \draw (a) to[out = 60, in = 60]     (b);
        \draw (a) to[out = 60, in = 60]     (c);
        \draw (a) to[out = 60, in = 60]     (d);

        \draw (w) to[out = 180, in = 180]   (x);
        \draw (w) to[out = 180, in = 180]   (y);
        \draw (w) to[out = 180, in = 180]   (z);
        \draw (w) to[out = 180, in = 240]   (a);
        \draw (a) to[out = 240, in = 240]   (b);
        \draw (b) to[out = 240, in = 240]   (c);
        \draw (b) to[out = 240, in = 240]   (d);

         \begin{scope}[on background layer]
             \fill[red!30!white] (0,0) -- (30:8.5cm) arc (30:150:8.5cm) -- (0,0);
             \fill[blue!30!white] (0,0) -- (150:8.5cm) arc (150:270:8.5cm) -- (0,0);
             \fill[green!30!white] (0,0) -- (30:8.5cm) arc (30:-90:8.5cm) -- (0,0);
         \end{scope}
    \end{tikzpicture}
    \caption{Illustration of the proof of \cref{thm:upper-bound-3-trees}.}
    \label{fig:3n/2}
\end{figure}

\begin{proof} 
In order to build a universal planar graph we use book embeddings of the given trees. A \emph{book embedding} of a graph $G$ is a decomposition of $G$ into outerplanar graphs that comes with an ordering $S$ of the vertices; the decomposition is such that each outerplanar graph has an embedding on a halfplane where the vertices lie on the boundary of the halfplane in the order $S$ (as on the spine of a book). The \emph{book thickness} of a graph is the smallest possible number of graphs in this decomposition. For example, every tree has a book embedding with thickness 1. 
    Let
    \begin{align*}
        S_1 &= (1, \ldots, n),\\
        S_2 &= \left( n, \ldots, \left\lfloor \frac{n}{2} \right\rfloor + 1, n+1, \ldots, \left\lfloor \frac{3n}{2} \right\rfloor \right),\\
        S_3 &=  \left( 1, \ldots, {\left\lceil \frac{n}{2} \right\rceil}, {n+1}, \ldots, \left\lfloor \frac{3n}{2} \right\rfloor \right).
    \end{align*}
    Since every tree has a book embedding with thickness 1, we can label the vertices of $T_i$ with $S_i$ such that this labelling yields a book embedding of $T_i$ with thickness 1. We take the union of the three trees, the resulting graph is planar. For an example see Figure~\ref{fig:3n/2}.
\end{proof}
We show in the following that there exist three caterpillars $T_1,T_2,T_3$ on $n$ vertices such that any planar graph containing all three caterpillars as a subgraph has at least $\frac{3}{2}n-o(n)$ vertices. In particular, setting $k\approx n^{1/3}$ and $\ell \approx n^{2/3}$ in the following theorem shows that there are trees $T_1,T_2,T_3$ such that any subgraph-universal planar graph containing all three trees has at least $\frac{3}{2}n-O(n^{2/3})$ vertices.
\begin{theorem}
    \label{thm:lower bound}
 Fix $1<k<\ell-2$ and $n$ large enough. Then every planar graph which contains $T_1=C_{n,1},T_2=C_{n,k}$ and $T_3=C_{n,\ell}$ as subgraphs has size at least
    \[
 \left( \frac{3}{2}-\frac{1}{k}-\frac{k}{\ell}\right) n-(k-1)(\ell-k)-\frac{\ell}{2}.
    \]

\end{theorem}

\begin{proof}
   Let $V(T_1)=\{x\} \cup L(x)$ where $x$ is the center of the star and $L(x)$ is the set of its leaves, $V(T_2)=\{y_1,\dots,y_k\} \cup (\cup_{i=1}^k L(y_i))$,$V(T_3)=\{z_1,\dots,z_\ell\}\cup(\cup_{i=1}^\ell L(z_i))$ where $Y=\{y_1,\dots,y_k\}$ (resp. $Z=\{z_1,\dots,z_{\ell}\}$) is the set of spine vertices of $T_2$ (resp. $T_3$) and $L(v)$ is the set of leaves of vertex $v$ in $T_2$ (or $T_3)$. 
   Consider a planar universal graph $G$ on $[N]$ for $T_1,T_2,T_3$. We assume $V(T_i)\subset [N]$ and $G=T_1 \cup T_2 \cup T_3$. We show that $N$ must be as large as claimed. 

 We can assume without loss of generality that $x=1$. Further, there exist (at least) $k-1$ vertices in $Y\setminus \{1\}$ and without loss of generality they are $2,\dots, k$. Further, there are (at least) $\ell-k$ vertices in $Z\setminus \{1,\dots,k\}$ and without loss of generality they are $k+1,\dots, \ell$. 
 Note that for every $2 \leq i \leq k$ and $k+1 \leq j \leq \ell$
   \begin{align}
   \label{eq:smallintersection}
       |L(1) \cap L(i) \cap L(j)|\leq 2 
   \end{align}
since otherwise $G$ contains a $K_{3,3}$-subgraph with one part $1,i,j$ and the other part consisting of three common neighbours of $1,i,j$ in $G$, and hence $G$ is not planar.

We let $L_y= L(2)\cup L(3) \cup \dots \cup L(k)$ and $L_z=L(k+1)\cup \dots \cup L(\ell)$ note that $|L_y|\geq \frac{(k-1)n}{k} - (k-1)$. At most $N-(n-1)$ of the vertices in $L_y$ can be outside of $L(1)$ since $|L(1)|=n-1$. Hence, 
\begin{align*}
   |L_y\cap L(1)| \geq \frac{k-1}{k} n-k +1-(N-n+1),
\end{align*}
and by the same arguments 
\begin{align*}
|L_z \cap L(1)| \geq \frac{\ell-k}{\ell} n-(\ell-k) -(N-n +1).
\end{align*}
By property \eqref{eq:smallintersection}, we get that
\begin{align*}
    n-1&\geq |L_y \cap L_z \cap L(1)|\\
    & \geq \left( \frac{k-1}{k}+\frac{\ell-k}{\ell}\right) n-2(N-n)-\ell -1-2(k-1)(\ell-k).
\end{align*}
Rearranging, this gives
\begin{align*}
    N & \geq \frac{1}{2}\left( \left( \frac{k-1}{k}+\frac{\ell-k}{\ell}+1\right) n-2(k-1)(\ell-k)-\ell\right)\\
     &=\left( \frac{3}{2}-\frac{1}{k}-\frac{k}{\ell}\right) n-(k-1)(\ell-k)-\frac{\ell}{2}.    \hfill \qedhere
\end{align*}
\end{proof}

\section{Planar universal graphs for (outer)planar graphs}\label{sec:planaruniversal}
Whereas there is a universal tree for $n$-vertex trees with $n^{\theta(\log (n))}$ vertices, we show that outerplanar and planar graphs behave differently. That is, an (outer)planar graph containing all $n$-vertex (outer)planar graphs has exponentially many vertices in $n$.
\planarlowerbound*

\begin{proof}
    Let $\mathcal{H}$ be the set of all $n$-vertex $4$-connected triangulations and let $G$ be a planar graph which is universal for $\mathcal{H}$. We consider a plane embedding of~$G$. Tutte~\cite{TutteCensus1962} showed that the number $k_n$ of $4$-connected triangulations is of the order of $n^{-5/2} (27/4)^n$.
    Note that given two $4$-connected triangulations $H_1,H_2$ from $\mathcal{H}$ in $G$ (and their induced embeddings) that share vertices, then either the vertices that they share appear on the outer face of both $H_1$ and $H_2$, or otherwise $H_1$ has to be contained in a bounded face of $H_2$ or vice versa (where the vertices of the face can be shared).
    As otherwise, if  $H_2$ has a vertex strictly inside and strictly outside a triangle of $H_1$ (or vice versa) then it has a cut set of size at most $3$. This would be a contradiction to the $4$-connectivity of $H_2$.
    
    We say that $H_1\leq H_2$ if $H_1$ is contained in a bounded face of $H_2$. This gives us a partial order on the set $\mathcal{H}$ which we extend to a linear order $H_1,\dots,H_{k_n}$ of $\mathcal{H}$. Note that removing all interior vertices of $H_{1},\dots,H_{i}$ in the embedding of $G$ leaves a graph $G'$ which is universal for $H_{i+1},\dots,H_{k_n}$. Hence
    \begin{align*}
        |V(G)| & \geq (n-3)\cdot k_n = \Omega(n^{-3/2} (27/4)^n). \hfill \qedhere
    \end{align*}
\end{proof}
We turn now to outerplanar graphs. We denote an outerplanar graph as \emph{path-like} when it is a maximal outerplanar graph and there is an embedding in which each triangle has an edge on the outer face. Note that the embedding with this property is unique up to isomorphism. Equivalently, the weak dual of a path-like outerplanar graph is a path. The \emph{weak dual} of a plane graph is a graph whose vertices are the bounded faces of the original graph such that two bounded faces are connected by an edge in the weak dual if they are adjacent in the original graph.

Let $\mathcal{H}_n$ be the class of path-like outerplanar graphs on $n$ vertices. 
\begin{theorem}
       An outerplanar graph containing all $n$-vertex outerplanar graphs has at least $\sqrt{2}^{n-5}$ vertices. 
\end{theorem}
\begin{proof}
An ear of a maximal outerplanar graph is a triangle with a vertex of degree $2$. Note that each path-like outerplanar graph with $n \geq 4$ vertices has exactly two ears. If $n=3$, $\mathcal{H}_n$ contains just a single triangle and for $n=4,5$ there is also a unique path-like outerplanar graph up to isomorphism (of the embedding). We give a lower bound on the pairs $(H,\Delta)$ where $H\in \mathcal{H}_n$ and $\Delta$ is an ear of $H$. 
Take the unique graph in $\mathcal{H}_5$ and $\Delta$ as one of its ears. We think of the outerplanar graph as rooted in $\Delta$. We can recursively glue a triangle to an edge on the outer face of the other ear, that is, the ear that is not $\Delta$. We have at each recursion the choice between two edges. This yields $2^{n-5}$ choices for $n\geq 6$ and each choice yields a different pair $(H,\Delta)$. 

We can assume that our universal outerplanar graph $G$ is maximal outerplanar, which means the number of triangles is $v(G)-2$, where $v(G) = |V(G)|$. By the pigeonhole principle there exists a triangle $\Delta$ of $G$ such that there are $2^{n-5}/(v(G)-2)$ distinct path-like outerplanar subgraphs $H$ for which $\Delta$ is an ear in $H$. Note that there is a bijection between path-like outerplanar graphs that have $\Delta$ as an ear and triangles at distance $n-3$ to $\Delta$ (where the distance between triangles is measured as the graph distance in the dual). But this means there are $2^{n-5}/(v(G)-2)$ distinct triangles of distance exactly $n-3$ to $\Delta$, hence $v(G)-2\geq 2^{n-5}/(v(G)-2)$, which by rearranging proves the theorem.
\end{proof}
We however found a polynomially sized planar graph which is universal for $\mathcal{H}_n$ from which we can construct a subexponential sized planar graph for all outerplanar graphs.
\begin{theorem}
    There exists a planar graph on $n^{\theta(\log(n))}$ vertices which contains every $n$-vertex outerplanar graph as a subgraph. 
\end{theorem}
We say a graph $G$ is a universal graph for $\mathcal{H}_n$ \emph{rooted at $e$} if $e$ is an edge in a triangle $\Delta$ of $G$ such that for every $H\in \mathcal{H}_n$ and for every ear $\Delta'$ in $H$ and $e'$ on the outer face of $\Delta'$ there exists a subgraph of $G$ and an isomorphism to $H$ which maps $\Delta$ to $\Delta'$ and $e$ to $e'$. 
\begin{lemma}
\label{lem:path-like}
    Let $n\geq 3$. There is a rooted universal planar graph $G_n$ for path-like $n$-vertex outerplanar graphs on less than $n^2$ vertices. 
\end{lemma}
\begin{proof}
We construct $G_n$ for $n\geq 3$ using two auxiliary graphs $G_n'$ and $G_n''$. We start with $G_n'$. Consider $P_{n}=v_1v_2\dots v_{n}$ a path of length $n$ and connect every second vertex, which gives the second power $P_{n}^2$ of the path.  Let $w_{i+2,0}=v_{i+2}$ for $1\leq i \leq n-2$ and stack $w_{i+2,j}$ into $v_{i}v_{i+1}w_{i+2,j-1}$ for $1\leq j \leq n-3-i$, see Figure~\ref{fig:universalforouterplanar}. We call the resulting graph $G'_n$. 

Now create the graph $G_n''$
in a similar way by starting with the path $x_1\dots x_n$ instead of $v_1 \ldots v_n$. Take the union of $G_n'$ and $G_n''$ and identify $x_1=u_2$, $x_2=u_1$ and $x_3=u_3$ and the $j$-th stacked vertex in the triangle $u_1u_2u_3$ with the $j$-th stacked triangle in $x_1x_2x_3$ for $1\leq j \leq n-4$. Let $G_n$ be the resulting graph, see Figure~\ref{fig:universalforouterplanar}. The number of vertices of $G_n$ is $3n-7+2\sum_{i=1}^{n-5} i <n^2$.
    
    We describe a process to generate all path-like outerplanar graphs rooted at an edge. Start with a triangle $u_1u_2u_3$ and fix the edge $u_1u_2$ in the triangle, which is our root. For each $4\leq i\leq n$ create a vertex $u_i$ and connect it with $u_{i-1}$ and one of the two neighbours of $u_{i-1}$. We have two choices at each step. In this way we generate every path-like outerplanar graph with a marked edge on one of its ears.

   Given a path-like outerplanar graph $H$ as above we assume that $u_4$ is connected to $u_2$ and $u_3$ and $|V(H)|\leq n$. By induction on $n$, we show that we can embed $H$ into $G_n'$ such that $u_1=v_1$ and $u_2=v_2$. Then it follows that we can also embed any path-like outerplanar graph for which $u_4$ is connected to $u_1,u_3$ into $G_n$ such that $u_1=v_1$ and $u_2=v_2$ by the symmetry of $G_n$. 
   
    For the base case, clearly we can embed the triangle into $G_n$ such that any edge of the triangle is matched to $v_1v_2$. Assume now that $n\geq 4$ and $H$ is not a triangle. Let $j$ be the smallest index such that $u_{j+2}$ is not adjacent to $u_2$ and if no such index exists set $j=|V(H)|-1$. Note that $3\leq j$ by the choice of $u_4$. By induction we can embed the path-like outerplanar graph induced by $u_2 u_{j}\dots u_n$ into $G_{n-1}'$ on the vertex set $v_2\dots v_{n}$ such that $u_2=v_2$, $u_{j}=v_3$ and $u_{j+1}=v_4$ (note that $u_{j+2}$ is adjacent to $u_{j}$ and $u_{j+1}$ so we are indeed in the same situation as before). Note that $u_j=w_{0,3}$ and set $u_{j-k}=w_{k,3}$ for $0\leq k \leq j-3$, $u_1=v_1$ and $u_2=v_2$.  This gives an embedding of $H$ into $G_n'$.
\end{proof}  

\begin{figure}
    \centering
    \begin{tikzpicture}
        \node at (0,0) {\includegraphics[width=0.3\textwidth,page=1]{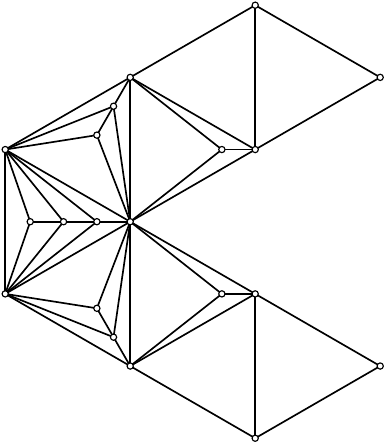}};

        \node at (5,0) {\includegraphics[width=0.3\textwidth,page=2]{figures/outerplanarinplanar-new.pdf}};
    \end{tikzpicture}
    \caption{The universal planar graph $G_7$ for all $7$-vertex path-like outerplanar graphs and the graph $G_7'$.}
    \label{fig:universalforouterplanar}
\end{figure}

Similarly to the Jordan-separator from \Cref{sec:universal_construction} we seek a way to split our graphs into small components. The following lemma follows by the Gyárfás path argument.
\begin{lemma}\cite{gayarfas}\label{lem:gayarfas}
    Let $T$ be a tree and let $v$ be a vertex of $T$. Then there exists a path $P$ in $G$ starting at $v$ such that the connected components of $G\backslash P$ contain at most $\frac{|V(G)|-1}{2}$ vertices each.
\end{lemma}
We call a path as in \Cref{lem:gayarfas} a \emph{Gyárfás path starting at $v$}. 
\begin{lemma}
     There is a universal planar graph for $n$-vertex outerplanar graphs with at most $n^{2\log_2(n)}$ vertices.  
\end{lemma}
\begin{proof}
   By \cref{lem:path-like} there exists a universal planar graph $G_n$ for path-like $n$-vertex outerplanar graphs rooted at an edge $e$ of size $n^2$. Let $\mathcal{G}_3$ be a triangle, $\mathcal{G}_4=G_4$ and $\mathcal{G}_5=G_5$. Construct $\mathcal{G}_n$ by appending $\mathcal{G}_{\lceil \frac{n}{2}\rceil}$ at each edge $e'$ in $G_n$ by identifying the special edge $e$ in $G_{\lceil \frac{n}{2}\rceil}$ (within $\mathcal{G}_{\lceil \frac{n}{2}\rceil}$) with $e'$. 

We show that any maximal outerplanar graph $H$ with an edge $e'$ on the outer face and on at most $n$ vertices can be embedded into $\mathcal{G}_n$ such that $e'$ is mapped to $e$. Consider the tree $T$ on $|V(H)|-2$ vertices which is the weak dual of $H$. Take a Gyárfás path $P$ in $T$ starting at the vertex whose triangle contains $e'$. Each connected component in $T\setminus P$ has less than $ \frac{n-2}{2}=\frac{n}{2} -1$ vertices and hence the maximal outerplanar graph corresponding to the subtrees has at most $\lceil \frac{n}{2}\rceil$ vertices. We embed the outerplanar graph corresponding to the Gyarfas path into $G_n$ and the outerplanar graph corresponding to the subtrees into the corresponding graphs $\mathcal{G}_{\lceil \frac{n}{2}\rceil}$ by induction. 

We show by induction that $\mathcal{G}_n$ has at most $n^{2\log_2(n)}$ vertices for $n\geq 3$. This is true for $n=3,4,5,6$ since $\mathcal{G}_3,\mathcal{G}_4,\mathcal{G}_5,\mathcal{G}_6$ have $3,5,8,40$ vertices, respectively. For $n\geq 7$ it holds that the number of vertices in $\mathcal{G}_n$ is at most the number of edges in $G_n$ times the number of vertices in $\mathcal{G}_{\lceil \frac{n}{2}\rceil}$ which is at most
\begin{align*}
    3n^2\left(\frac{n+1}{2}\right)^{2\log_2(\frac{n+1}{2})} &\leq 12(n+1)^{2\log_2(n+1)-2} \leq n^{2\log_2(n)}
\end{align*}
    where the last inequality holds for $n\geq 4$.
\end{proof}

\section{Computational results}
\label{sec:computer}

We used a Sagemath script
to search for $n$-vertex planar graphs that are universal for $n$-vertex trees for small values of~$n$.
Apparently, our script
found universal stacked triangulations
on $n \le 20$ and $n=22$ vertices, which was our initial hint towards \Cref{thm:polynomial_bound}.
We provide the script and the examples found by the computer as supplemental data~\cite{supplemental_data_anonymous}. 
Note that the memory requirements and computing times 
grow exponential in the number of vertices as all triangulations are stored in RAM. 
For $n=15$ the example can be found in about 5 CPU minutes on an average laptop.
However, to find examples for $n=16,\ldots,20$ and $n=22$ (see \Cref{fig:stacked22})
we used a computing cluster.
The example for $n=22$ was found after about 7 CPU days. The computations for $n=21$ timed out after 8 CPU days. 
Moreover for $n \geq 48$ there is no universal graph for all trees.
\begin{figure}
    \centering\includegraphics[scale=0.5]{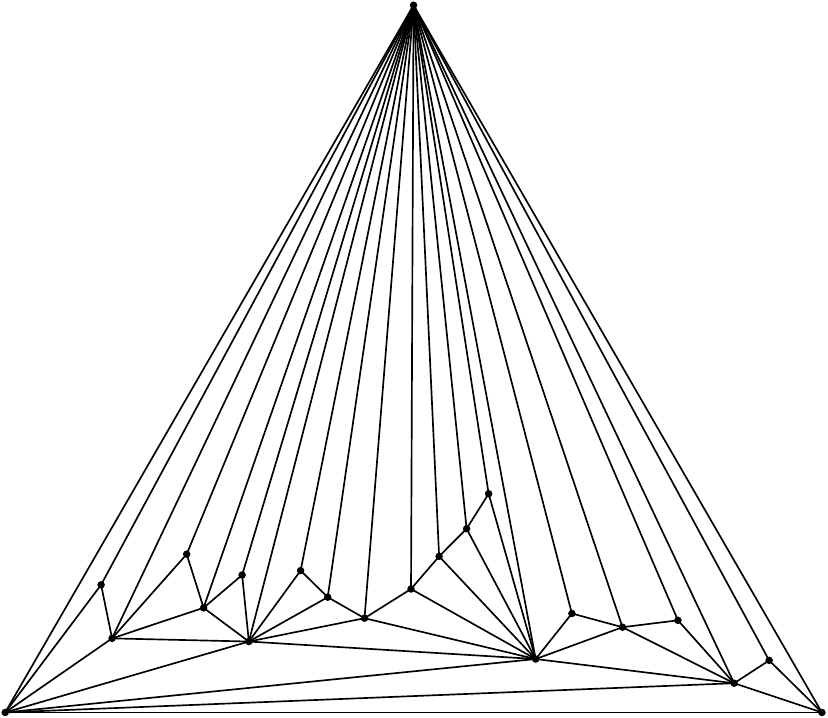}
    \caption{A stacked triangulation on 22 vertices that is universal for all trees on 22 vertices.}
    \label{fig:stacked22}
\end{figure}

\begin{example}
\label{ex:threetrees}
There is no planar graph on $48$ vertices containing $T_1 = C_{48,1}$, $T_2 = C_{48,2}$, and $T_3 = C_{48,8}$ as subgraphs.
\end{example}

\begin{proof}
Suppose $G$ is a planar graph on $48$ vertices containing $T_1,T_2,T_3$ as subgraphs. We fix for each $T_i$ a subgraph of $G$ isomorphic to $T_i$ and when we say $T_i$ we mean this specified copy of $T_i$ in $G$. Let $[48]$ be the vertex set of $G$. Let $u_1 \in [48]$  be the center vertex of the star $T_1$,  $v_1,v_2 \in [48]$ be the spine vertices of $T_2$ and $w_1,\dots, w_{8}\in[48]$ be the spine vertices of $T_3$. First suppose $u_1\neq v_1,v_2$.
Note that 5 vertices of $w_1,\dots,w_{8}$ are not $u_1,v_1,v_2$. Assume $w_1 \neq u_1,v_1,v_2$. Note that $w_1$ has degree at least $6$, hence it has three neighbours $z_1,z_2,z_3$ which are different from $u_1,v_1,v_2$. But then $G$ contains a $K_{3,3}$-minor with one part being $\{u_1\},\{v_1,v_2\},\{w_1\}$ and the other part being $\{z_1\},\{z_2\},\{z_3\}$.

Hence we can assume without loss of generality that $u_1=v_1$. Further let $A=N(u_2)\setminus \{v_1\}$ be the set of at least $23$ common neighbours of $v_1$ and $u_2$. To avoid a $K_{3,3}$ minor, every vertex $w_i\neq v_1,u_2$ can have at most $2$ neighbours in $A$. Deleting $v_1,u_2$ from $T_3$ can partition the spine of $T_{3}$ into at most $3$ parts. To avoid a $K_{3,3}$, each part of the spine can have at most $2$ of their leaves on the spine and the total number of leaves of these parts is at least $6\cdot 5=30$. Hence $24$ leaves have to be outside  $A$, in particular, one of $v_1,u_2$ is a leave of $T_3$. Hence, deleting $v_1,u_2$ cuts the spine of $T_3$ into at most two parts, at most $4$ of the $35$ leaf vertices can be in $A$, hence $31$ vertices have to be outside $A$, which is not possible.
\end{proof}

\paragraph{How the script works}
First, the script uses the graph enumerator nauty \cite{McKayPiperno14}
to generate the list of all $n$-vertex trees. 
We then use the planar graph enumerator plantri \cite{BrinkmannMcKay2007}
to iterate over all possible candidates
and test whether it is universal for all trees.
Since finding a subgraphs is a notoriously hard problem (at least as hard as graph isomorphism),
we employ a SAT solver.
Given a tree $T=(V_T,E_T)$ and a planar graph $G=(V_G,E_G)$ on~$[n]$,
we create a SAT instance as following.
To find an injection $\pi:V_T \to V_G$
such that for every edge $\{a,b\} \in E_T$ of the tree 
the graph contains the edge $\{\pi(a),\pi(b)\} \in E_G$,
we use the variables $X_{a,b}$ to indicate $\pi(a) = b$,
and add for every $\{a,b\} \in E_T$ with $\{p,q\} \not\in E_G$
a clause $\neg X_{a,p} \vee \neg X_{b,q}$ that prevents 
that adjacent vertices $a$ and $b$ are simultaneously 
mapped to a non-adjacent pair $p$ and~$q$.
To solve the CNF instance, we use the python interface pysat to the solver picosat.
To test whether a planar graph is a stacked triangulation, 
we used a recursive routine which tries to remove stacked vertices. 
Our script can be run with the command
\verb|sage enum_universal_sat.sage [n]|.
We also provide an additional parameter \verb|-s| that restricts the search to stacked triangulations.

\section{Concluding remarks}
\label{sec:concluding}

In this paper we considered subgraph-universal planar graphs for triples of trees on~$n$ vertices and for all trees on~$n$ vertices. 
It would be interesting to investigate the minimum number of vertices of a subgraph-universal planar graph for all possible sets of $k$ trees on~$n$ vertices in dependence of $k$.
Studying universal graphs for a constant number of trees might help to understand universal graphs for all $n$-vertex trees.
In particular, we wonder whether there is a universal graph with a quadratic number of vertices. 

Extending the study of minor-free graph classes, a natural question which arises is whether there is a polynomial size $K_{t+1}$-minor-free universal graph for the class of $K_t$-minor-free graphs. Although we have not found a polynomially sized universal planar graph for outerplanar graphs, allowing $K_{3,3}$'s could help in significantly reducing the size of the universal planar graph we found. In particular, starting with the constructing for path-like outerplanar graphs in \cref{lem:path-like}, one could add edges between vertices which are stacked in adjacent copies of the power of the path without creating a $K_5$. On the other hand it is likely that there is no subexponential $K_t$-minor-free universal graph for $n$-vertex $K_t$-minor free graphs and this can potentially be shows using similar arguments as for the lower bound for planar universal graph for planar graphs in \cref{thm:planarlowerbound}.

Our construction for a universal outerplanar graph for $n$-vertex trees is not induced-universal because any new Jordan separator is always connected to the previous Jordan separator (this edge is sometimes needed). However, we can embed each $n$-vertex tree in our construction in such a way that we can reconstruct it from its vertex set. We wonder whether the number of vertices of an induced-universal planar graph for all $n$-vertex trees is polynomial or superpolynomial in $n$.

\bibliographystyle{plainurl}
\bibliography{references}

\end{document}